\documentclass[11pt,oneside]{amsart}
\usepackage{amsmath,ifthen, amsfonts, amssymb,
srcltx, amsopn, textcomp}

\usepackage{hyperref}
\usepackage{graphicx}

\usepackage[small]{caption}
\usepackage{marginnote}

\newcommand{\showcomments}{yes}
\renewcommand{\showcomments}{no}

\newtheorem{thm}{Theorem}[section]
\newtheorem{lem}[thm]{Lemma}

\newtheorem{prop}[thm]{Proposition}

\theoremstyle{definition}

\newtheorem{rem}[thm]{Remark}
\newtheorem{exmp}[thm]{Example}

\newtheorem{prob}[thm]{Problem}

\newcommand{\field}[1]{\mathbb{#1}}
\newcommand{\integers}{\ensuremath{\field{Z}}}
\newcommand{\rationals}{\ensuremath{\field{Q}}}
\newcommand{\naturals}{\ensuremath{\field{N}}}
\newcommand{\reals}{\ensuremath{\field{R}}}

\newcommand{\sslash}{{/\mkern-6.5mu/}}

 \label{mainB}

\DeclareMathOperator{\lcm}{lcm}

\newsavebox{\commentbox}
\newenvironment{com}%
{\ifthenelse{\equal{\showcomments}{yes}}%
{\footnotemark
        \begin{lrbox}{\commentbox}
        \begin{minipage}[t]{1.25in}\raggedright\sffamily\tiny
        \footnotemark[\arabic{footnote}]}
{\begin{lrbox}{\commentbox}}}%
{\ifthenelse{\equal{\showcomments}{yes}}%
{\end{minipage}\end{lrbox}\marginpar{\usebox{\commentbox}}}
{\end{lrbox}}}

\newcommand{\Gonna}{}

\begin{document}
\title{Residually Finite Tubular Groups}
\author[N.~Hoda]{Nima Hoda}
           \address{Dept. of Math. \& Stats.\\
                    McGill Univ. \\
                    Montreal, QC, Canada H3A0B9 }
           \email{nima.hoda@mail.mcgill.ca}
\author[D.~T.~Wise]{Daniel T. Wise}
           \address{Dept. of Math. \& Stats.\\
                    McGill Univ. \\
                    Montreal, QC, Canada H3A0B9 }
           \email{wise@math.mcgill.ca}
\author{Daniel J. Woodhouse}
      \address{Dept. of Math.\\
               Technion \\
               Haifa 32000, Israel}
               \email{woodhouse.da@technion.ac.il}
\subjclass[2010]{20E26}
\keywords{Residually Finite, Tubular Groups}
\date{\today}
\thanks{Research supported by NSERC. The third named author was supported by the Israel Science Foundation (grant 1026/15).}

\begin{abstract}
 A tubular group $G$ is a finite graph of groups with $\mathbb{Z}^2$ vertex groups and $\mathbb{Z}$ edge groups.
 We characterize  residually finite tubular groups: $G$ is residually finite if and only if its edge groups are separable.
 Methods are provided to determine if $G$ is residually finite.
 When $G$ has a single vertex group an algorithm is given to determine residual finiteness.
  \end{abstract}

\maketitle

\section{Introduction}

A f.g.\ group $G$ is \emph{tubular} if it splits as a finite graph of groups with $\integers$ edge groups and $\integers^2$ vertex groups.
A group $G$ is \emph{residually finite} if for each nontrivial $g\in G$,
there is a finite quotient of $G$ so that the image of $g$ is nontrivial.
The goal of this paper is to determine which tubular groups are residually finite.

The case where $G$ is a single HNN extension was handled by
Andreadakis, Raptis and Varsos 
\cite{AndreadakisRaptisVarsos88}. 
However the full complexity of the situation is not apparent for a single HNN extension, as  residual finiteness  coincides with virtual specialness whereas failure of residual finiteness coincides with a problematic Baumslag-Solitar subgroup.
Kim~\cite{Kim04} proved that having \emph{isolated cyclic subgroups} is a sufficient condition for residual finiteness.
In the language of this paper, isolated cyclic subgroups translates to saying the tubular group is \emph{primitive}.

\begin{com}
{\bf Problem:} Are there residually finite tubular groups that are quasi-isometric to non-residually finite tubular groups?
\end{com}

\subsection{Quick Survey of Results about Tubular Groups}
Tubular groups form a class of seemingly straightforward groups 
that are increasingly recognized as a surprisingly rich source of diverse behavior.
Burns, Karass, and Solitar gave the first example of a f.g.\ 3-manifold group that is not subgroup separable, and their example arises as a tubular group \cite{BKS87}.
Croke and Kleiner used this same tubular group to show that the boundary of a CAT(0) space is not an invariant of CAT(0) groups~\cite{CrokeKleiner2000}.
Gersten gave a tubular group as an example of a free-by-cyclic group that does not act 
properly and semi-simply on a CAT(0) space \cite{Gersten94}.
Wise gave an example of a tubular group that is CAT(0) but not Hopfian \cite{Wise96}.
Brady and Bridson~\cite{BradyBridson2000} characterised the Dehn functions of \emph{snowflake groups}, a subclass of tubular groups, to show that there are f.p.\ groups with isoperimetric functions $n^d$ where $d\in D$ is a dense subset of $[2,\infty)$.
Gardam and Woodhouse showed that certain Snowflake groups embed as finite index subgroups of one-relator groups~\cite{GardamWoodhouse17}, and Button observed that many of these groups are not residually finite~\cite{Button}.
Cashen gave a quasi-isometric classification of tubular groups \cite{Cashen2010}.
Wise gave a criterion for a tubular group to be cubulated \cite{WiseGerstenRevisited}.
Button showed that if a tubular group is free-by-cyclic, then it is cubulated \cite{Button17}.
Woodhouse classified which cubulations are finite dimensional and showed that a tubular group is virtually special if and only if it acts freely on a finite dimensional CAT(0) cube complex \cite{WoodhouseTubularAction, Woodhouse15VirtuallySpecial}.

\subsection{Statement of Main Result}
A f.g.\ group $G$ is \emph{tubular} if it splits as a finite graph of
groups with $\integers^2$ vertex groups and $\integers$ edge groups.
A tubular group $G$ is \emph{primitive} if each edge group is a
maximal cyclic subgroup of its vertex groups, and hence of $G$.  A
nontrivial element $(a,b) \in \integers^2$ is \emph{primitive} if
$\gcd(a,b) = 1$, that is $(a,b)$ is not a ``proper power.''

There are two goals to this paper. 
The first is to characterize which tubular groups are residually finite, and the second is to provide practical means of deciding the question.
The following theorem, addressing the first goal, is a special case of a more extensive characterization given in
Theorem~\ref{thm:main equivalence}. 

\begin{thm}
 A tubular group is residually finite if and only if it is virtually primitive.
 \end{thm}
 
 Although we are not able to settle the question of decidability in general, in the motivating case, where $G$ has a single vertex group and a single edge group we are able to provide the following, a consequence of Proposition~\ref{prop:singleVertexClassification} and Lemma~\ref{lem:tupalgo} in Section~\ref{sec:OneVertexSolution}.
 
 \begin{thm} \label{thm:decidability}
  Let $G$ be a tubular group with a single vertex group. Then there is an algorithm that decides in finite time if $G$ is residually finite or not. 
 \end{thm}
 
 To address tubular groups in general, 
 we introduce the \emph{expansion sequence} for a tubular group, which we motivate in the following subsection.

\subsection{Two Illustrative Examples} \label{subsec:ExpansionExamples}

The expansion sequence for a tubular group is nontrivial, even in the simple case of a graph of groups with a single vertex group and two edge groups.
Given a tubular group $G = G_0$ the expansion sequence is a series of homomorphisms.
At the $i$-th stage of the computation we obtain a tubular group $G_i$ and a homomorphism $G_{i-1} \rightarrow G_i$. 
The sequence $G = G_0 \rightarrow G_1 \rightarrow G_2 \rightarrow  \cdots$ is the \emph{expansion sequence}.
We are presented with a dichotomy: either the expansion sequence \emph{terminates} or it continues indefinitely, that is to say it is \emph{non-terminating}.
By Lemma~\ref{lem:primitive implies terminates}, a terminating expansion sequence is equivalent to $G$ being residually finite.
Ideally, we would like to determine if an expansion sequence is non-terminating after a finite number of steps.
The simplest way to verify this is if the sequnce starts repeating itself.
We call such sequences \emph{recurrent}.
Unfortunately, not all non-terminating expansion sequences are recurrent.
See Example~\ref{exmp:nonRecurrentExample}.
We conjecture however that if a tubular group is not residually finite, then some subtubular group will have an expansion sequence that repeats itself.

We give two examples of such computations to illustrate and motivate what will be happening in this paper.

\begin{exmp} \label{exmp:terminates}
 The tubular group $G$ below splits over a graph with a single vertex group and two edge groups.
 The elements $(1,0)$ and $(0,1)$ generate the vertex group $G_v = \integers^2$ and $s$ and $t$ are the stable letters
 associated to the edge groups.
  \[G \ = \  \langle \  \integers \times \integers, s, t \ \mid \ (1,0)^s = (2,2),\  (0,1)^t = (1,1) \ \rangle\]

 $G$ is not primitive since 
 $(2,2)$
  is not primitive in $G_v$.
 Note that $t$ conjugates a primitive element to a primitive element.
 We will construct a homomorphism $G\rightarrow G'$ from $G$ to another tubular group $G'$ with the same underlying graph, such that vertex and edge groups map injectively, and such that the stable letter $s$ conjugates a pair of primitive elements in $G'$.
 A simple way to do this is to add the element $(\frac12,0)$ to the vertex group and extend the conjugation by $s$ linearly so that $(\frac12,0)$ is conjugated to $(1,1)$.
We thus obtain the following new tubular group: 
\[G'  \ = \ \langle\  \textstyle\frac12\integers\times \integers, \; s, t \  \mid \  \;  (\textstyle{\frac12},0)^s = (1,1), \;  (0,1)^t  = (1,1)
 \ \rangle.\]
 
 There is a homomorphism $G \rightarrow G'$ that maps $(0,1)$, $ (1,0)$, $s$, and $t$ to themselves in $G'$.
 This morphism is  the \emph{expansion map}.
 As $G'$ is a primitive tubular group we say that we have found a \emph{primitive target} for $G$, which implies by Theorem~\ref{thm:main equivalence} that $G$ is residually finite.
\end{exmp}

\begin{exmp}\label{exmp:no terminate}
Consider the following tubular group $G$ having a single vertex group and two edge groups. 
 Note that $G$ is almost identical to the group in Example~\ref{exmp:terminates}, with a slight adjustment to the elements conjugated to $(1,0)$ and $(0,1)$.
 \[G \  = \ \langle \  \integers\times\integers, \ s, t \ \mid \  (1,0)^s = (2,4),  \ (0,1)^t = (1,2) \ \rangle.\]
 $G$ is not primitive since $(2,4)$ is not primitive in the vertex group $G_v$. All other images of the edge group generators are primitive.
 As in Example~\ref{exmp:terminates} we will construct an ``expansion map'' by adding the element $(\frac12,0)$ to $G_v$
 and extending the conjugation by $s$ linearly so that $s$ conjugates $(\frac12,0)$ to $(1,2)$ .
We thus obtain the tubular group $G'$ below,
and obtain a  homomorphism $G \rightarrow G'$ mapping $\integers\times \integers$ and $s$ and $t$ identically to themselves.
 \[G' = \langle \ \textstyle{\frac{1}{2}}\integers\times\integers, \, s, t \ \mid
  \ 
 (\textstyle{\frac12}, 0)^s = (1,2), \;  (0,1)^t = (1,2) \ \rangle\]
 
 Unfortunately, $G'$ is not primitive. Indeed,
 $(1,2)$ is no longer primitive since $(1,2) = 2(\frac12,1)$.
 We may then construct another expansion map.
 This time however, in order to extend both conjugations linearly we need to include the elements $(\frac14,0)$ and $(0,\frac12)$.
 We thus obtain the tubular group 
 \[G'' = \langle \  \textstyle{\frac14}\integers \times \textstyle\frac12\integers, \ s, t \  \mid \
 (\textstyle{\frac14}, 0)^s=(\textstyle{\frac12},1) ,
  \ (0,\textstyle{\frac12})^t= (\textstyle{\frac12},1) \ \rangle\]
 and the expansion map $G' \rightarrow G''$.
 This time the expansion map has not improved our situation at all
 since $G''$ is isomorphic to $G'$. 
 The isomorphism is given by scaling both $(\frac14,0)$ and $(0,\frac12)$ by $2$.
 Repeating this process yields $G'''$ which is again isomorphic to $G'$ and therefore we will never arrive at a primitive target.
 This situation is a \emph{recurrent expansion sequence} and by Lemma~\ref{lem:primitive implies terminates} it implies that $G$ is not residually finite.
\end{exmp}
 
 In Examples~\ref{exmp:terminates}~and~\ref{exmp:no terminate}, the vertex group is of the form
 $\frac1n\integers\times\frac1m\integers$ at each stage. However, the algorithm generally wanders through groups that are not subdirect products of cyclic groups commensurable with the factors
 of the initial product decomposition.

 \subsection{Structure of this paper}
 
 In Section~\ref{sec:Morphisms} we define a range of algebraic constructions that we will use to characterize residually finite tubular groups in Section~\ref{sec:ScalingMorphisms} and
 Section~\ref{sec:ETuples}.
 Section~\ref{sec:expansion algorithm} defines the expansion sequence of a tubular group and provides a general framework for understanding residual finiteness of tubular groups.
 Section~\ref{sec:SnowflakeGroups}  applies the techniques of Section~\ref{sec:expansion algorithm} to the snowflake groups of Brady and Bridson~\cite{BradyBridson2000}, to determine their residual finiteness and recover a result of Button.
Section~\ref{sec:OneVertexSolution} shows that residual finiteness is decidable when the tubular group has a single vertex group.


\section{Morphisms and Primitivity} \label{sec:Morphisms}

Let us establish the notation we use for the spitting of a group $G$ as a graph $\Gamma$ of groups.
We assume $\Gamma$ is directed, we let $\mathcal{E}\Gonna$ and $\mathcal{V}\Gonna$ denote its sets of edges and vertices, and for an edge $e\in \mathcal{E}\Gonna$, we let $-e$ and $+e$ denote its  initial and terminal vertices. For $e\in \mathcal{E}\Gonna$ we let $G_e$ denote
the associated edge group, and for $v\in\mathcal{V}\Gonna$ we let $G_v$ denote its vertex group.
For each edge $e$, let $\varphi^{\pm}_e : G_e \rightarrow G_{\pm e}$ denote the two inclusions
 of an edge group into its vertex groups.

Let $G$ and $G'$ be groups which split over the graphs $\Gamma$ and $\Gamma'$ respectively.
A \emph{morphism} of graphs of groups is a homomorphism $f: G \rightarrow G'$ such that there is a morphism $f_*: \Gamma \rightarrow \Gamma'$ of undirected graphs, and the restriction of $f$ to vertex or edge groups gives homomorphisms $f_v: G_v \rightarrow G_{v'}'$ and $f_e: G_e \rightarrow G_{e'}'$ where $v' = f_*(v)$ and $e' = f_*(e)$.

A \emph{rigid morphism} $f: G \rightarrow G'$ is a morphism such that $f_*$ is an isomorphism and each $\phi_v$ and $\phi_e$ is injective.

 A tubular group $G$ has a \emph{primitive target} if there a rigid morphism
  $\overline{f}: G \rightarrow \overline{G}$ between tubular groups
   such that $\overline{G}$ is primitive.
Similarly, a tubular group $G$ has a \emph{primitive domain} if there is a rigid morphism $\underline{f}: \underline{G} \rightarrow G$ such that 
$\underline{G}$ is a primitive tubular group.

The following holds by the definitions:

\begin{lem} \label{lem:standardMorphismConstruction}
 Let $G$ and $G'$ be tubular groups that split over the same underlying graph $\Gamma$.
 Suppose that $G_v' \leqslant G_v$ and $G_e' \leqslant G_e$ and that the edge maps of $G'$ are restrictions of the edge maps of $G$.
 Then there is a rigid morphism $\phi: G' \rightarrow G$ induced by the inclusion maps on the vertex and edge groups.
\end{lem}

\subsection{Local Quotients}\label{defn:local quotient}
 Let $f: G' \rightarrow G$ be a rigid morphism of tubular groups with underlying graph $\Gamma$.
 Suppose that for the edge inclusions $\varphi_e' : G'_e \rightarrow G'_v$ and $\varphi_e : G_e \rightarrow G_v$ we have
 $f \circ \varphi_e'(G'_e) = \varphi_e(G_e)\cap f(G'_v)$. 
 Note that this equality always holds when $G'$ is primitive.
Define a group $G\sslash G'$ that splits 
over $\Gamma$ as follows:
 \begin{enumerate}
  \item $(G\sslash G')_v = G_v / f(G_v')$,
  \item $(G\sslash G')_e = G_e / f(G_e')$,
  \item Attaching maps $(G\sslash G')_e\rightarrow (G\sslash G')_v$ are projections of 
 $G_e\rightarrow G_v$,
  \item There is a morphism $q:G\rightarrow G\sslash G'$ that is induced by the quotient maps 
  $G_v \rightarrow (G\sslash G')_v$ and $G_e \rightarrow (G\sslash G')_e$.
 \end{enumerate}
 
 Each map $G_e/G_e' \rightarrow G_v/G_v'$ is injective,
since if $g\in G_e$ maps to the identity in $G_v/G_v'$ then
the image of $g$ in $G_v$ lies in $f(G_v')$.
But then $g \in \varphi_e(G_e) \cap f(G_v')$ so $f(g) \in G_e'$ by hypothesis.
Hence $g$ represents the identity in $G_e/G_e'$.

 Having verified the injectivity of  attaching maps of edge groups of $G\sslash G'$ 
we see that the data for $G\sslash G'$ actually yields a splitting over $\Gamma$.
The induced morphism $q: G \rightarrow G\sslash G'$ is 
the \emph{local quotient} of $f$.

\section{Regulating \texorpdfstring{$\mathcal{E}$}{E}-tuples} \label{sec:ETuples}

\newcommand{\kk}{\underline{k}}

Let $G$ be a tubular group.  Let $\kk = (k_e)_{e \in \mathcal{E}}$ be
an $\mathcal{E}$-tuple of integers, one for each edge of $G$.  For
each edge group $G_e$ let $G_e^{(\kk)} = k_eG_e$.  For each vertex
group $G_v$ let $G_v^{(\kk)} \leq G_v$ be the subgroup generated by
the inclusions of the $G_e^{(\kk)}$ under the attaching maps.  An
$\mathcal{E}$-tuple $\kk$ is \emph{regulating} if for each edge $e$
and generator $g_e \in G_e$, the element $\varphi^{\pm}_e(k_eg_e)$ is
primitive in $G^{(\kk)}_{\pm e}$.

\begin{rem}
  \label{rem:etupcp} Let $\kk$ be an $\mathcal{E}$-tuple.  Then, for
  any positive integer $n$, we have that $\kk$ is regulating if and
  only if $n\kk = (n k_e)_{e \in \mathcal{E}}$ is regulating.  So, in
  searching for regulating $\mathcal{E}$-tuples, it suffices to
  consider those $\kk = (k_e)_{e \in \mathcal{E}}$ having
  $\gcd(k_e : e \in \mathcal{E}) = 1$.
\end{rem}

\begin{lem}
  \label{lem:goodtup} Let $G$ be a tubular group.  Then $G$ has a
  primitive domain if and only if $G$ has a regulating
  $\mathcal{E}$-tuple.
\end{lem}
\begin{proof}
  Suppose $G$ has a regulating $\mathcal{E}$-tuple $\kk$.  Extend each
  $G_v^{(\kk)}$ to a rank $2$ subgroup $\bar G_v^{(\kk)}$ of $G_v$
  such that $G_v^{(\kk)}$ is a maximal subgroup of its rank in
  $\bar G_v^{(\kk)}$.  The inclusions
  $G_e^{(\kk)} \hookrightarrow G_e$ and
  $\bar G_v^{(\kk)} \hookrightarrow G_v$ induce a rigid morphism
  $G^{(\kk)} \to G$.  Each edge group $G^{(\kk)}_e$ is generated by
  $k_eg_e$ where $g_e$ is a generator of $G_e$.  Since
  $\varphi^{\pm}_e(k_eg_e)$ is primitive in $G^{(\kk)}_{\pm e}$ and so
  in $\bar G^{(\kk)}_{\pm e}$, the image
  $\varphi^{\pm}_e(G^{(\kk)}_e)$ is a maximal cyclic subgroup of
  $\bar G^{(\kk)}_{\pm e}$.  Hence $G^{(\kk)} \to G$ is a primitive
  domain.

  Suppose $G$ has a primitive domain $G' \to G$.  Let
  $\kk = (k_e)_{e \in \mathcal{E}}$ be an $\mathcal{E}$-tuple where
  $k_e = [G_e : G'_e]$.  Then
  $\varphi_e^{\pm}(G^{(\kk)}_e) = \varphi_e^{\pm}(k_eG_e) =
  \varphi_e^{\pm}(G'_e)$ is a maximal cyclic subgroup of
  $G'_{\pm e} < G_{\pm e}$ and so $\varphi_e^{\pm}(G^{(\kk)}_e)$ is a
  maximal cyclic subgroup of $G^{(\kk)}_{\pm e} < G'_{\pm e}$.  Then
  if $g_e$ generates $G_e$ then $\varphi_e^{\pm}(k_eg_e)$ generates
  $\varphi_e^{\pm}(G^{(\kk)}_e)$ and so $\varphi_e^{\pm}(k_eg_e)$ is
  primitive in $G^{(\kk)}_{\pm e}$.
\end{proof}

\section{Scaling Morphisms, Naive Morphisms, and Primitivity} \label{sec:ScalingMorphisms}
Given $H \cong \mathbb{Z}^n$ and a nonzero rational number $\alpha\in \rationals^*$,
it is natural to define the group $\alpha H$, and likewise to define $\alpha h$ when $h\in H$.
This is justified by noting that there is a unique inclusion $H \hookrightarrow \rationals^n$ up to conjugation by $GL_n(\rationals)$.

Let $G$ be a tubular group with underlying graph $\Gamma$.
 Let $G_e = \langle g_e \rangle$ and $G_v = \langle a_v, b_v \rangle$.
For $\alpha \in \rationals^*$ we  define the tubular group $\alpha G$ with underlying graph $\Gamma$ as follows:
 The vertex and edge groups of $\alpha G$ are
 \[
  \alpha G_v = \langle \alpha a_v, \alpha b_v \rangle \; \textrm{ and } \; \alpha G_e = \langle \alpha g_e \rangle.
 \]
 Its edge inclusions are determined by linear extension: $\phi_e^{\pm}(\alpha g_e) = \alpha \phi_e^{\pm}(g_e)$.

 Note that  $\alpha G$ is primitive when $G$ is primitive.
The \emph{scaling morphism}
 is a rigid isomorphism $G \rightarrow \alpha G$ induced by $g\mapsto \alpha g$ for each $g$ in a vertex or edge group.

We will also employ the following two rigid morphisms
 that arise when $\alpha=n\in \naturals$ and $\alpha=\frac{1}{n}$ respectively:
They map each vertex group and edge group to the obvious copies of itself within the target.

The \emph{first naive morphism} $f: nG \rightarrow G$ is defined since $nG_v \leqslant G_v$ and $nG_e \leqslant G_e$ for all vertices $v \in \mathcal{V}\Gonna$ and $e \in \mathcal{E}\Gonna$. 
The inclusions of the vertex and edge groups extend to a rigid morphism by Lemma~\ref{lem:standardMorphismConstruction}.

The \emph{second naive morphism} $g: G \rightarrow \frac{1}{n}G$ is defined since $G_v \leqslant \frac{1}{n}G_v$ and $G_e \leqslant \frac{1}{n} G_e$ for all vertices $v \in \mathcal{V}\Gonna$ and $e \in \mathcal{E}\Gonna$.
The inclusions of the vertex and edge groups extend to a rigid morphism since edge maps are extended linearly so 
 Lemma~\ref{lem:standardMorphismConstruction} applies.

We emphasize that the scaling morphism $G\rightarrow \frac{1}{n} G$ 
and the second naive morphism $G\rightarrow \frac{1}{n} G$ are different for $n>1$
even though they have the same domain and target.
The scaling morphism is an isomorphism and restricts to a scaling isomorphism on each vertex and edge group. The second naive morphism restricts to an inclusion of a subgroup on each vertex and edge group.
There is likewise a $\frac1n$ scaling isomorphism $nG \rightarrow G $ which differs from
the first naive morphism.

\begin{lem} \label{lem:PrimitiveTargetAndDomainEquivalence}
 $G$ has a primitive target if and only if $G$ has a primitive domain.
\end{lem}
\begin{proof}
 If $G$ has a primitive target then there is a morphism $\overline{f}:  G \rightarrow \overline{G}$.
 Let $n_e = [\overline{G}_e:G_e]$ and $n_v = [\overline{G}_v : G_v]$.
 Let $n = \lcm\{ n_e, n_v \mid e\in \mathcal{E}\Gonna, v \in \mathcal{V}\Gonna\}$.
 Then $n \overline{G}$ is also primitive, and there is the naive morphism $F: n \overline{G} \rightarrow \overline{G}$.
 It follows from our choice of $n$ that $n \overline{G}_v \leqslant G_v \leqslant \overline{G}_v$ for $v \in \mathcal{V}\Gonna$, and $n \overline{G}_e \leqslant G_e \leqslant \overline{G}_e$ for $e \in \mathcal{E}\Gonna$.
 Therefore, by Lemma~\ref{lem:standardMorphismConstruction} there is a morphism $\underline{f}: n\overline{G} \rightarrow G$ induced by inclusion of the edge groups such that $\underline{f} \circ \overline{f}$ gives the inclusion of $n \overline{G}_v$ into $\overline{G}_v$ for all $v \in \mathcal{V}\Gonna$, and similarly for all edge groups.
 Hence $F = \underline{f} \circ \overline{f}$ and $G$ also has a primitive domain.

If $G$ has a primitive domain then there is a tubular group $\underline{G}$ and a morphism $\underline{f}: \underline G \rightarrow {G}$.
 Let $m_e = [{G}_e: \underline G_e]$ and $m_v = [{G}_v : \underline G_v]$.
 Let $m = \lcm\{ n_e, n_v \mid e\in \mathcal{E}\Gonna, v \in \mathcal{V}\Gonna\}$.
 Then $\frac{1}{m}\underline{G}$ is a primitive tubular group, and there is the naive morphism $F: \underline{G} \rightarrow \frac{1}{m}\underline{G}$.
 It follows from our choice of $m$ that $\underline{G}_v \leqslant G_v \leqslant \frac{1}{m} \underline{G}_v$ for $v \in \mathcal{V}\Gonna$, and $\underline{G}_e \leqslant G_e \leqslant \frac{1}{m} \underline{G}_e$ for $e \in \mathcal{E}\Gonna$.
 Therefore, by Lemma~\ref{lem:standardMorphismConstruction}, there is a morphism $\overline{f}: G \rightarrow \frac{1}{m}\underline{G}$
  induced by the inclusions of edge groups such that $\underline{f} \circ \overline{f}$ gives the inclusion of $\underline{G}_v$ into $\frac{1}{m}\underline{G}_v$ for all $v \in \mathcal{V}\Gonna$, and similarly for all edge groups.
 Hence $F = \underline{f} \circ \overline{f}$ and $G$ also has a primitive target.
\end{proof}

A subgroup $H\subset G$ is \emph{separable} if $H$ is the intersection of finite index subgroups of $G$.
The following is well-known:
\begin{lem} \label{lem:separability}
 \leavevmode
 \begin{enumerate}
  \item \label{second part} The intersection of separable subgroups of $G$ is separable.
  \item \label{first part} A maximal abelian subgroup  $A\leqslant G$ 
  of a residually finite group is separable.
 \end{enumerate}
\end{lem}
\begin{proof}
Statement~\eqref{second part} follows from the definition.
Statement~\eqref{first part} holds as follows:
If $g\notin A$, then there exists $a\in A$ such that $k=gag^{-1}a^{-1}\neq 1$.
By residual finiteness, there is a finite quotient $\phi:G\rightarrow G'$ such that
$\phi(k)\neq 1$. Let $A' \leqslant G'$ be a maximal abelian subgroup 
containing $\phi(A)$, and note that $\phi(g)\not\in A'$.
Then $A$ lies in the finite index subgroup $\phi^{-1}(A')$, but $g\notin\phi^{-1}(A')$.
\end{proof}

\begin{thm}\label{thm:main equivalence}
 The following are equivalent:
 \begin{enumerate}
  \item \label{cond1} $G$ is residually finite,
  \item \label{cond4} $G$ has a primitive domain,
  \item \label{cond5} $G$ has a primitive target,
  \item \label{cond2} $G$ is virtually primitive,
  \item \label{cond3} $G$ has separable edge groups.
  \item \label{cond6} $G$ has a regulating $\mathcal{E}$-tuple.
 \end{enumerate}
\end{thm}

\begin{proof}
$(\ref{cond1}\Rightarrow\ref{cond3})$\ 
Each conjugate of a vertex group is a maximal abelian subgroup, and hence separable by 
Lemma~\ref{lem:separability}.\eqref{first part}.
Each edge group is the intersection of conjugates of incident vertex groups.
Hence the edge group is separable by Lemma~\ref{lem:separability}.\eqref{second part}.
 
$(\ref{cond3} \Rightarrow \ref{cond2}) $\
By the separability of each edge group $G_e$, there is a finite index subgroup $J^e \leqslant G$ such that
$G_e \leqslant J^e$ and $G_e$ is a direct factor of each of its vertex groups in $J^e$.
Let $G'=\cap_e J^e$. 
Then $G'$ is primitive.

$(\ref{cond2} \Rightarrow \ref{cond1})$\
Since being virtually residually finite is equivalent to being residually finite, we will just show
that primitive implies residually finite.
Let $G$ be a primitive  tubular group.
For each $n$, consider the morphism $nG \rightarrow G$ and its associated local quotient $q_n : G \rightarrow G \sslash nG$. 
As $G \sslash nG$ is a graph of finite groups, it is virtually free and hence residually finite.
Therefore it suffices to show that for each nontrivial $g \in G$ there exists $n$ such that $q_n(g)$ is nontrivial.

Either $g$ is elliptic or $g$ is hyperbolic with respect to the action on the associated Bass-Serre tree.
If $g$ is elliptic we can assume, after conjugation, that $g \in G_v$ regarded as $(p,q) \in \mathbb{Z}^2$.
Choose $n > \max\{ |p|, |q| \}$.
Then $q_n(g)$ is nontrivial.
If $g$ is hyperbolic, then it has a normal form without any backtrack.
We will explain how to choose $n$ such that $q_n$ also has a normal form without any backtrack.
Each potential backtrack is of the form $t^{\pm 1} h t^{\mp 1}$ for some stable letter $t$ and $h \in G_v$.
Let $G_e$ be the edge group associated to $t$, and note that $h \notin G_e$.
By primitivity, $G_v = G_e \times \mathbb{Z} \cong \mathbb{Z}^2$ with $(1,0)$ the generator of $G_e$.
Since $h \notin G_e$, we have $h = (p,q)$ with $q \neq 0$.
Hence, this potential backtrack is not a backtrack whenever $n > |q|$.
Choosing $n$ to satisfy this condition for each potential backtrack guarantees that $q_n(g)$ is nontrivial.

$(\ref{cond4} \Leftrightarrow \ref{cond5})$\
This is Lemma~\ref{lem:PrimitiveTargetAndDomainEquivalence}.

 $ (\ref{cond4} \Rightarrow\ref{cond2})$\ 
 Let $\underline{f}: \underline{G} \rightarrow G$ be the primitive domain for $G$.
Let $G \sslash \underline{G}$ be the associated local quotient.
If $G$ has an edge group generator $g_e$ which has a proper root $\frac{1}{k}g_e \in G$, then $\frac{1}{k}g_e$ maps to a torsion element in $G \sslash \underline{G}$. 
%
Note that $G \sslash \underline{G}$ is virtually free as a graph of finite groups
\cite{KarrassPietrowskiSolitar73}. 
Let $F \leqslant G \sslash \underline G$ be a finite index free subgroup, and let 
$G' \leqslant G$ be the preimage of $F$ in $G$.
Finally, observe that $G'$ is primitive as any proper root of an edge generator in $G'$ would map to a torsion elements in $G \sslash \underline{G}$.

$(\ref{cond2} \Rightarrow \ref{cond4})$
Since finite index subgroups of primitive tubular groups are primitive,
there also exists a finite index normal subgroup $G' \leqslant G$ such that $G'$ is primitive.
 The induced splitting of $G'$ shows that $G'$ is also tubular, so inclusion $G' \hookrightarrow G$ is a morphism 
 of tubular groups.
 Let $p: \Gamma' \rightarrow \Gamma$ be the morphism of graphs associated to the inclusion.
Let $v\in \mathcal{V}\Gonna$. If $u',v' \in p^{-1}(v)$, then as $G'$ is a normal subgroup the vertex groups $G'_{u'}$ and $G'_{v'}$ have identical images inside $G_v$.
 The analogous statement holds for each edge $e\in \mathcal{E}\Gonna$.
 
 We construct $\underline{G}$ from $G'$ as follows:
 The vertex group $\underline{G}_v$ is the image of $G'_{v'}$ in $G_v$ for some and hence any choice $v' \in p^{-1}(v)$.
 The edge group $\underline{G}_e$ is the image of $G'_{e'}$ in $G_e$ for some and hence any choice $e' \in p^{-1}(e)$.
 The edge group inclusions of $G'$ determine the edge group inclusions of $\underline{G}$.
 By Lemma~\ref{lem:standardMorphismConstruction} we get a rigid morphism $\underline F : \underline G \rightarrow G$ determined by the inclusions of the vertex and edge groups.
 As $G'$ is a primitive tubular group, $\underline G$ is also a primitive tubular group.

$(\ref{cond4} \Leftrightarrow \ref{cond6})$\
This is Lemma~\ref{lem:goodtup}.
\end{proof}

\section{The Expansion Sequences}\label{sec:expansion algorithm}

\newcommand{\depth}{\text{depth}}

Let $G$ be a tubular group with underlying graph $\Gamma$ and tubular space $X$.
For each edge $e \in \mathcal{E}\Gonna$ fix a choice of generator $g_e$ of $G_e$.
The \emph{degree} $d_e^{\pm}$ of an attaching map $\varphi_e^{\pm}$ is
the order of the torsion factor in $G_{\pm e} / \phi_e^{\pm}(G_e)$.
Let $d_e = \lcm\{d_e^+, d_e^-\}$.
We refer to the tuple $(d_e)_{e \in \mathcal{E}}$ as the \emph{edge degrees}.

Define a tubular group $G'$ with underlying graph $\Gamma$ as follows:
The edge group $G_e' = \frac{1}{d_e} G_e$ and the vertex group $G_v' = \langle G_v, H_v \rangle$, where 
\[ H_v = \Big\{ \frac{1}{d_e} \phi_e^{+}(g_e) \mid e \in \mathcal{E}\Gonna, + e = v \Big\} \cup \Big\{ \frac{1}{d_e} \phi_e^{-}(g_e) \mid e \in \mathcal{E}\Gonna, - e = v \Big\}.\]

As $\frac{1}{d_e}\phi_e^{+}(g_e) \in G_v'$, for all $e \in \mathcal{E}\Gonna$ such that $+e = v$, we obtain the edge map $\phi_e'^+ : G_e' \rightarrow G_v'$ by extending $\phi_e^+$ linearly.
The  inclusions $p_v : G_v \rightarrow G_v'$ and $p_e: G_e \rightarrow G_e'$ determine a rigid morphism $p: G \rightarrow G'$  called the \emph{expansion morphism}.
An expansion is \emph{trivial} if it is the identity map. This occurs precisely when $G$ is primitive.

The following lemma shows that there is a bound on the complexity of the tubular group produced by the expansion morphism.

\begin{lem} \label{lem:edgeExpansionSpectrumBound}
 Let $G$ be a tubular group and $(d_e)_{e\in \mathcal{E}}$ the edge degrees.
 Let $\ell = {\lcm\{d_e\mid e \in \mathcal{E}\Gonna \}}$.
 Let $G \rightarrow G'$ be the expansion morphism, and $(d_e')_{e \in \mathcal{E}}$ be the edge degrees of $G'$.
 Then $d_e'$ divides $\ell$ for all $e \in \mathcal{E}$.
\end{lem}

 \begin{proof}
  Let  $v = +e$.
  Let $K \leqslant G_v$ be the maximal cyclic subgroup of $G_v$ containing $\phi_e^+(G_e)$.
  Then $d^+_e$ is the order of the quotient $K / \phi_e^+(G_e)$.
  Let $K' \leqslant G_v'$ be the maximal cyclic subgroup of $G_v'$ containing ${\phi'}_e^+(G_e')$.
  The claim follows by showing that $\ell$ is divided by the order of $K' / {\phi'}_e^+(G_e')$.
  
  First note that $K = \langle \frac{1}{d_e^+} \phi_e^+(g_e) \rangle \leqslant \langle \frac{1}{d_e} \phi_e^+(g_e) \rangle = {\phi'}_e^+(G_e')$.
  Second note that $G_v' = \langle G_v, \frac{1}{d_e} \phi_e^+(g_e), \ldots \rangle \leqslant \frac{1}{\ell}G_v$, so $K' \leqslant \frac{1}{\ell} K$.
Together this implies that $K \leqslant {\phi'}_e^+(G_e') \leqslant K' \leqslant \frac{1}{\ell} K$ so the order of $K' / {\phi'}^+_e(G_e')$ is a factor of $\ell$.
 \end{proof}

\begin{lem} \label{lem:edgeExpansionFactorisation}
Let $p: G \rightarrow G'$ be the expansion map.
 If $G$ has a primitive target $\overline{f}: G \rightarrow \overline{G}$, then  $\overline{f}$ factors as
 $\overline{f} = \overline{p} \circ p$
 for some morphism $\overline{p} : G' \rightarrow \overline{G}$.
 \end{lem}

\begin{proof}
The vertex and edge groups of $G$ can be viewed as subgroups of the corresponding vertex and edge groups of both $\overline{G}$ and $G'$.
We  deduce that $G_e \leqslant G_e' \leqslant\overline{G}_e$ as $\overline{G}$ is primitive, so $\frac{1}{d_e} G_e$ must be a subgroup of $\overline{G}_e$.
Similarly, $G_v \leqslant G'_v \leqslant\overline{G}_v$ as the primitivity of $\overline{G}$ implies that $\frac{1}{d_e}\phi^{\pm}_e(g_e)$ must be in $\overline{G}_v$ for $v = \pm e$.
Therefore, by Lemma~\ref{lem:standardMorphismConstruction} there exists a rigid morphism $\overline{p}: G' \rightarrow \overline{G}$ such that $\overline{p} \circ p = \overline{f}$.
%
\end{proof}

\begin{lem} \label{lem:expansionSurjection}
 Let $p: G \rightarrow G'$ be an expansion map.
 Then $p(G) = G'$.
\end{lem}

\begin{proof}
 Recall that for each edge $e \in \mathcal{E}$ we fixed a generator $g_e$. 
 We then let $d_e^{\pm}$ be the degree of the attaching map $\phi_e^{\pm}$, and $d_e = \lcm\{d_e^+, d_e^-\}$.
 Then $G_e' = \frac{1}{d_e} G_e$, and $G_{\pm e}'$
 were defined to include the element $\frac{1}{d_e}\phi_e^{\pm}(g_e)$,
 for all incident edges $e$.
 Note that $\frac{1}{d_e^\pm} \phi_e^{\pm}(g_e)$ was already an element of $G_{\pm e}$, since $d_e^\pm$ was the order of the torsion factor in $G_{\pm e} / \phi^{\pm}_e(G_e)$.
 Therefore $\frac{1}{d_e^{\pm}} g_e$ and thus $\frac{1}{d_e} g_e$ will be in the image of $p$.
 It then follows that $G_e'$ is contained in $p(G)$ for all edges $e$, and therefore $G_v'$ is contained in $p(G)$ for all $v \in \mathcal{V}$.
\end{proof}

An \emph{expansion sequence} is a sequence of nontrivial expansions 
\[
 G \rightarrow G_1 \rightarrow G_2 \rightarrow \cdots \rightarrow G_t \rightarrow \cdots 
\]

The following asserts that 
a finite expansion sequence is equivalent to residual finiteness.

\begin{lem}\label{lem:primitive implies terminates}
 If $G$ has a primitive target then any expansion sequence starting with $G$ 
 has length bounded by
 $\sum_e [\overline{G_e}:G_e]$.
  
Conversely, if the expansion sequence 
$G\rightarrow 
\cdots \rightarrow G_t$  
\emph{terminates} in the sense that it cannot be extended,
then $G_t$ is primitive, and hence $G$ has a primitive target.
 \end{lem}
\begin{proof}
 Let $\overline{f}: G \rightarrow \overline{G}$ be a primitive target for $G$.
By Lemma~\ref{lem:edgeExpansionFactorisation}, $\overline{f}$ factors through the map
$G\rightarrow G_m$ for each $m$.
 Therefore, the sum of the degrees of each edge group $G_e$ in $\overline{G}_e$ provides an upper bound on the length of a sequence of edge expansions.

The composition $G=G_1\rightarrow G_t=\overline{G}$
yields the converse.
For if $G_t$ is not primitive then $d_e \neq 1$ for some edge $e$. Hence there is a nontrivial expansion of $G_t$.
\end{proof}

The expansion sequence is computable so Lemma~\ref{lem:primitive implies terminates} shows that there is an algorithm which can find a primitive target, should one exist.  Specifically, the algorithm would perform edge expansions until the expansion sequence terminates.
An effective algorithm would also need to identify when $G$ is non-residually finite.
Suppose that $G \rightarrow G_1 \rightarrow G_2 \rightarrow\cdots$ is a non-terminating expansion sequence. Then we say the expansion sequence is \emph{recurrent} if $G_i$ is isomorphic to $G_j$ via some rigid isomorphism, for some $i < j$.
Therefore if either a terminating or a recurrent expansion sequence could be found in finite time, the question of residual finiteness would be algorithmically decidable.
Unfortunately, in general, there are non-residually finite tubular groups with non-recurrent, infinite expansion sequences.

\begin{exmp} \label{exmp:nonRecurrentExample}
 Let
 \[
  G = G_0 = \langle \mathbb{Z} \times \mathbb{Z}, s, t \ \mid \ s(1,0)s^{-1} = (2,0), \ t(0,1)t^{-1} = (1,1) \rangle.
 \]

 There is a single non-primitive vector $(2,0)$ among the relations so the first edge expansion is given by dividing the first edge group by two to obtain
 
 \[
  G_1 = \langle \textstyle{ \frac{1}{2}} \mathbb{Z} \times \mathbb{Z}, s,t \ \mid \ s(\textstyle{\frac{1}{2}},0)s^{-1} = (1,0), \ t(0,1)t^{-1} = (1,1) \rangle.
 \]

 Observe that the elements $(0,1)$ and $(1,1)$ remain primitive in $\textstyle{\frac{1}{2}} \mathbb{Z} \times \mathbb{Z}$, so the only non-primitive element in the relations is $(1,0)$.
 Therefore 
 the $n$-th term in the  expansion sequence is:
 \[
  G_n = \langle \textstyle{\frac{1}{2^n}} \mathbb{Z} \times \mathbb{Z}, s,t \mid s(\textstyle{\frac{1}{2^n}},0)s^{-1} = (\frac{1}{2^{n-1}},0), t(0,1)t^{-1} = (1,1) \rangle
 \] 
 Thus the expansion sequence does not terminate so $G$ is not residually finite.
 But $G_n \neq G_m$ for $n \neq m$.
 Indeed, since all maximal rank 2 free abelian groups in $G_n$ are conjugate to the vertex group $\frac{1}{2^n} \mathbb{Z} \times \mathbb{Z}$, we can assume that an isomorphism $G_n \rightarrow G_m$ sends $\frac{1}{2^n} \mathbb{Z} \times \mathbb{Z}$ to $\frac{1}{2^m} \mathbb{Z} \times \mathbb{Z}$.
 Any conjugate of the vertex group in $G_n$ that nontrivially intersects $\frac{1}{2^n} \mathbb{Z} \times \mathbb{Z}$ does so in a cyclic subgroup $\langle (\frac{1}{2^n}, 0) \rangle$, $\langle (\frac{1}{2^{n-1}}, 0) \rangle$,
 $\langle (0, 1) \rangle$, and $\langle (1,1) \rangle$.
 Similarly, in $G_m$ nontrivial intersections of conjugates of the vertex group intersect $\frac{1}{2^m} \mathbb{Z} \times \mathbb{Z}$ in the cyclic subgroups $\langle (\frac{1}{2^m}, 0) \rangle$, $\langle (\frac{1}{2^{m-1}}, 0) \rangle$,
 $\langle (0, 1) \rangle$, and $\langle (1,1) \rangle$.
 By identifying $G_n \cong \mathbb{Z}^2$ we can compute the \emph{unsigned intersection numbers} of these cyclic subgroups.
 The unsigned intersection number of $\langle (p,q) \rangle $ and $\langle (r , s) \rangle$ is
 the absolute value of the determinant of the matrix $ \begin{pmatrix} p & r \\
                       q & s \\
       \end{pmatrix}$.
 The unsigned intersection number is invariant up to multiplication by elements of $GL_2(\mathbb{Z})$.
 So, as any isomorphism $G_n \rightarrow G_m$ must send conjugates of vertex groups to conjugates of vertex groups, the unsigned intersection numbers are an invariant of $G_n$.
 The largest intersection number of $G_n$ is $2^n$ and is achieved by the vectors $(0,1)$ and $(1,1)$, which are identified with $(0,1)$ and $(2^n, 1)$ when $(G_n)_v$ is identified with $\mathbb{Z}^2$.
 Therefore $G_n$ is not isomorphic to $G_m$ if $n \neq m$.

 

 
 Note that if we consider the subtubular group 
 \[
  G' = G'_0 = \langle \mathbb{Z} \times \mathbb{Z}, s \ \mid \ s(1,0)s^{-1} = (2,0) \rangle.
 \]
 Then we can compute that
 \[
  G'_1 = \langle \textstyle{\frac{1}{2}}\mathbb{Z} \times \mathbb{Z}, s \ \mid \ s(\textstyle{\frac{1}{2}},0)s^{-1} = (1,0) \rangle.
 \]
 As there is a rigid isomorphism $G'_1 \rightarrow G'$ we deduce that $G'$ is recurrent.
\end{exmp}
 
Consideration of examples and computer experiments leads to the following:
  \begin{prob}
   Does every non-residually finite tubular group contain a subtubular group with recurrent expansion sequence?
  \end{prob}

The following example illustrates that even a terminating expansion sequence can be arbitrarily long for a fixed graph $\Gamma$.

\begin{exmp}\label{exmp:termlength}
For each $n$ let $G^{(n)}$
be the tubular group presented by:
 \[ \langle \mathbb{Z} \times \mathbb{Z}, t \mid t(1,0)t^{-1} = (2, 2^n) \rangle\] 
 $G^{(n-1)}$  is isomorphic to the expansion of $G^{(n)}$ which has the following presentation:
 \[\langle \textstyle{\frac12} \mathbb{Z} \times \mathbb{Z}, t \mid t (\textstyle{\frac12}, 0) t^{-1} = (1,2^{n-1}) \rangle
 \]
We thus have the terminating expansion sequence 
 $G^{(n)} \rightarrow G^{(n-1)} \rightarrow \cdots \rightarrow G^{(1)} \rightarrow G^{(0)}$.
 So the expansion sequence of $G^{(n)}$ has length $n+1$.
\end{exmp}

\begin{com}The linearity section is commented out for now\end{com}

 \section{The Residually Finite Snowflake Groups} \label{sec:SnowflakeGroups}
\emph{Snowflake groups} are the following tubular groups 
 for positive integers $p\geq q$:
 \[
  G_{pq} = \langle \mathbb{Z}^2 ,s,t \mid (q,0)^s = (p,1), (q,0)^t = (p,-1) \rangle 
 \]
Brady and Bridson showed that  $G_{pq}$ has Dehn function $\simeq n^{2\alpha}$ for ${\alpha = \log_2(\frac{2p}{q}})$ in \cite{BradyBridson2000}.
Gardam and Woodhouse showed that many snowflake groups are finite index subgroups of one-relator groups  \cite{GardamWoodhouse17}.
  This provided examples of non-automatic one-relator groups that 
do not contain Baumslag-Solitar subgroups of the form $BS(m,n)=
\langle a,t \mid (a^m)^t=a^n\rangle$ with $m\neq \pm n$.
Subsequently, Button observed that some of these one-relator groups are CAT(0) but not residually finite and has classified the residually finite snowflake groups~\cite{Button}. 
We now reproduce his classification using our method.  
  \begin{thm}
    $G_{pq}$ is residually finite if and only if $q$ divides $2p$.
  \end{thm}
  \begin{proof}
  If $q = 1$ then $G_{pq}$ is a primitive tubular group and hence residually finite by Theorem~\ref{thm:main equivalence}.
  If $q >1$ then we perform the expansion map $G_{pq} \rightarrow G_{pq}'$ where each edge group is divided by $q$. 
 The  vertex group of $G_{pq}'$ is:
  \[
   \bigg\langle \ \Bigl(\frac{p}{q}, \ \frac{1}{q}\Bigr),\  \Bigl(\frac{p}{q},\  -\frac{1}{q}\Bigr),\  (1,0),\ (0,1) \  \bigg\rangle
   \]

We swap the components of these generators, scale them by $q$ and set them as the rows of a matrix below.  We obtain a two element basis by performing integer row operations to reduce the matrix to Hermite normal form:
   
   \[
    \begin{bmatrix}
     1 & p \\
     -1 & p \\
     0 & q \\
     q & 0 \\
    \end{bmatrix}
   \; \rightarrow \;
   \begin{bmatrix}
    1 & p \\
    0 & 2p \\
    0 & q \\
    0 & -qp \\
   \end{bmatrix}
   \; \rightarrow \;
   \begin{bmatrix}
    1 & p \\
    0 & \gcd(2p, q)\\
    0 & 0 \\
    0 & 0 \\
   \end{bmatrix}
   \]
Thus $G_{pq}'$ has the following presentation:
  \[
   G_{pq}' = \left\langle  \ \Big(\frac{\gcd(q,2p)}{q}, 0 \Big), \ \Big(\frac{p}{q}, \frac{1}{q} \Big), \  s,t \; \ \Big| \; \Big(1,0\Big)^s = \Big(\frac{p}{q}, \frac{1}{q}\Big), \ \Big(1,0\Big)^t = \Big(\frac{p}{q}, \  \frac{-1}{q}\Big) \  \right\rangle.
  \]

  If $q\mid 2p$ then $G_{pq}'$ is primitive and hence $G_{pq}$ is
  residually finite, by Lemma~\ref{lem:primitive implies terminates}
  and Theorem~\ref{thm:main equivalence}.  Otherwise, $G_{pq}'$ is not
  primitive and has a nontrivial expansion map where each edge group
  is divided by the degree of the torsion factor in
  \[ \Big\langle \Big(\frac{\gcd(q,2p)}{q}, 0 \Big), \Big(\frac{p}{q},
    \frac{1}{q}\Big) \Big\rangle \ \Big/ \ \left\langle \Big(1,0\Big)
    \right\rangle.\] Since the vertex group in $G_{pq}'$ is generated
  by the elements conjugated by $s$ and $t$ we deduce that the
  expansion map $G_{pq}' \rightarrow G_{pq}''$ is a scaling morphism
  and therefore an isomorphism.  Thus, the expansion sequence is
  recurrent if $q \nmid 2p$ and so $G_{pq}$ is not residually finite
  by Lemma~\ref{lem:primitive implies terminates} and
  Theorem~\ref{thm:main equivalence}.
\end{proof}
 
\section{Deciding Residual Finiteness for Single Vertex Group} \label{sec:OneVertexSolution}

Let $G$ be a tubular group with a single vertex group $G_v$.  We will
show that the problem of determining the residual finiteness of $G$ is
decidable.

\begin{prop}
  \label{prop:singleVertexClassification}
  Let $G$ be a tubular group with a single vertex group $G_v$.  Assume
  that $G$ has at least two edges and that
  $\bigl\langle \varphi_e^{+}(G_e), \varphi_e^{-}(G_e) \bigl\rangle <
  G_v$ has rank 2 for every edge $e$.  Let $e_1, \ldots, e_n$ be the
  edges in the underlying graph of $G$.  Let $u_i, v_i \in G_v$
  correspond to the generators of the cyclic subgroups of $G_v$
  conjugated by the stable letter associated to $e_i$.  Let
  $t_i \in \rationals_{>0}$ be minimal such that
  $t_iu_i \in \langle u_{i+1}, v_{i+1} \rangle$, where the indices are
  considered modulo $n$.  Let $\kk = (k_e)_{e \in \mathcal{E}}$ be
  given by $k_{e_i} = k_i$ and write $\kk = (k_1, k_2, \ldots, k_n)$.
  If $\kk$ is regulating then
  \[ \kk = \bigg( m, m\frac{z_1}{t_1}, \ldots, m \frac{z_1z_2 \cdots
      z_{n-1}}{t_1t_2 \cdots t_{n-1}} \bigg) \] for some
  $m,z_1,z_2,\ldots,z_n \in \integers$ for which 
 $z_1 \cdots z_n = t_1 \cdots t_n $.
\end{prop}
\begin{proof}
  Suppose $\kk$ is regulating.  Then
  \[ \frac{k_{i+1}t_i}{k_i} k_i u_i \; = \; k_{i+1}t_iu_i \; \in \;
    k_{i+1} \langle u_{i+1}, v_{i+1} \rangle \; \leqslant \;
    G^{(\kk)}_v.\] Since $k_iu_i$ is primitive in $G^{(\kk)}_v$ by the
  definition of regulating, we deduce that
  $\frac{k_{i+1}t_i}{k_i} = z_i$ for some integer $z_i \in \integers$.
  Hence $\frac{k_{i+1}}{k_i} = \frac{z_i}{t_i}$ and so
  \[\frac{z_1\cdots z_n}{t_1 \cdots t_n} 
  \ =  \
  \frac{k_2}{k_1}\frac{k_3}{k_2}\cdots \frac{k_n}{k_{n-1}}\frac{k_1}{k_n}
 \ = \   1.\]
 Setting $m = k_1$ we recover the claim. 
\end{proof} 

We apply Proposition~\ref{prop:singleVertexClassification} in the
following example.

\begin{exmp}
  Let $G$ be the tubular group with the following presentation.
  \[ \bigl\langle \integers \times \integers, s, t \mid \ (2,-4)^s =
    (-1,-2), \ (-6,-6)^t = (2,2) \bigr\rangle \] Following
  Proposition~\ref{prop:singleVertexClassification}, let
  $u_1 = (2,-4), \ v_1 = (-1,-2), \ u_2 = (-6,6), \ v_2 = (2,2)$, and
  compute that $t_1 = 2$ and $t_2 = \frac{4}{3}$.  Since $t_1 t_2$ is
  not an integer, there do not exist integers $z_1$ and $z_2$ such
  that $z_1 z_2 = t_1 t_2$. Hence,
  Proposition~\ref{prop:singleVertexClassification} implies that $G$
  has no regulating $\mathcal{E}$-tuple.  Hence $G$ is not residually
  finite, by Theorem~\ref{thm:main equivalence}.
\end{exmp}

\begin{exmp}
  Recall that the snowflake group $G_{pq}$ is the tubular group
  presented by
  \[ \bigl\langle \mathbb{Z}^2 ,s,t \mid (q,0)^s = (p,1), (q,0)^t =
    (p,-1) \bigr\rangle \] for positive integers $p \ge q$.  Following
  Proposition~\ref{prop:singleVertexClassification}, let
  $u_1 = u_2 = (q,0), \ v_1 = (p,1), \ v_2 = (p,-1)$, and compute that
  $t_1 = t_2 = 1$.  Then, by
  Proposition~\ref{prop:singleVertexClassification} and
  Remark~\ref{rem:etupcp}, there is a regulating $\mathcal{E}$-tuple
  for $G_{pq}$ if and only if $(k_1, k_2) = (1,1)$ is a regulating
  $\mathcal{E}$-tuple.  That is, if and only if $(q,0)$, $(p,1)$ and
  $(p,-1)$ are primitive in the subgroup
  $H = \bigl\langle (q,0), (p,1), (p,-1) \bigr\rangle$ that they
  generate.  If \[ r(p,\pm 1) = a(q,0) + b(p, \mp 1) \] for some
  $r \in \rationals$ and $a,b \in \integers$ then
  $r = -b \in \integers$ and so $(p, \pm 1)$ is always primitive in
  $H$.  On the other hand \[ r(q,0) = a(p,1) + b(p, -1) \] holds for
  some $r \in \rationals$ and $a,b \in \integers$ if and only if
  $a = b$ and $r = \frac{2p}{q}a$.  Hence $(q,0)$ is primitive in $H$
  if and only if $q | 2p$.  Thus we see that $G_{pq}$ has a regulating
  $\mathcal{E}$-tuple if and only if $q | 2p$.
\end{exmp}

Theorem~\ref{thm:decidability} follows from Theorem~\ref{thm:main
  equivalence} and the following lemma.

\begin{lem}
  \label{lem:tupalgo} Let $G$ be a tubular group with a single vertex group
  $G_v$.  There is an algorithm which determines if $G$ has a
  regulating $\mathcal{E}$-tuple.
\end{lem}
\begin{proof}
  The algorithm first checks to see if the images
  $\varphi_e^{\pm}(G_e)$ of any edge group $G_e$ are commensurable but
  distinct in $G_v$.  In such a case we have
  $\varphi_e^{+}(k_e g_e) = q \varphi_e^{-}(k_e g_e)$ for some
  $q \in \rationals - \{1\}$ where $g_e$ is a generator of
  $G_e$.  Then the $\varphi_e^{\pm}(k_eg_e)$ cannot both be primitive
  in any subgroup of $G_v$ so no $\mathcal{E}$-tuple is regulating and
  the algorithm may return a ``no'' answer.

  Henceforth we assume that if $\varphi_e^{\pm}(G_e)$ are
  commensurable for some $e \in \mathcal{E}$ then they are equal.  Let
  $G'$ be the subtubular group obtained from $G$ by removing an edge
  $e$ for which the $\varphi_e^{\pm}(G_e)$ are equal.  Given a
  regulating $\mathcal{E}'$-tuple $\kk'$ for $G'$ we may obtain a
  regulating $\mathcal{E}$-tuple $\kk$ for $G$ as follows.  If
  $G^{(\kk')}_v \cap \varphi_e^{\pm}(G_e)$ is trivial then we obtain
  $\kk$ by extending $\kk'$ with any $k_e \in \integers - \{0\}$.
  Otherwise, let $q \in \rationals_{>0}$ be minimal such that
  $q\varphi_e^{\pm}(G_e) < G^{(\kk')}_v$ and choose
  $m \in \integers - \{0\}$ such that $mq \in \integers$.  We obtain
  $\kk$ by extending $m\kk' = (mk'_e)_{e \in \mathcal{E}'}$ with
  $k_e = mq$.

  Thus the algorithm discards all edges $e$ for which the
  $\varphi_e^{\pm}(G_e)$.  If $G$ has a single edge group $G_e$ then
  any $k_e \in \integers - \{0\}$ gives a regulating $\kk$ and
  so the algorithm returns a ``yes'' answer in this case.

  At this point in the algorithm $G$ has at least two edges and for
  each edge $e$ the $\varphi_e^{\pm}(G_e)$ are not commensurable.  By
  Proposition~\ref{prop:singleVertexClassification} and
  Remark~\ref{rem:etupcp}, we need only consider finitely many
  integers $z_1, \ldots, z_n$ and $m$ to check if $G$ has a regulating
  $\mathcal{E}$-tuple.  For each $z_1, \ldots, z_n$ and $m$ we compute
  the corresponding $G^{(\kk)}_v$ and determine whether the $k_iu_i$
  and $k_iv_i$ are primitive in $G^{(\kk)}_v$.
\end{proof}

\bibliographystyle{plain}
\bibliography{wise,ref2}
\end{document}